\documentclass[12pt]{amsart}

\usepackage{amssymb} 
\usepackage{amsthm}
\usepackage{amscd}
\newtheorem{thm}{Theorem}[section]
\newtheorem{lem}[thm]{Lemma}
\newtheorem{lem-dfn}[thm]{Lemma-Definition}
\newtheorem{prop}[thm]{Proposition}

\theoremstyle{definition}
\newtheorem{defn}[thm]{Definition}

\newtheorem{ex}[thm]{Example}

\newtheorem*{acknowledgement}{Acknowledgement}
\theoremstyle{remark}
\newtheorem{clm}{Claim}
\newtheorem{rem}[thm]{Remark}
\numberwithin{equation}{section}

\newcommand{\thmref}[1]{Theorem~\ref{#1}}
\newcommand{\lemref}[1]{Lemma~\ref{#1}}

\newcommand{\proref}[1]{Proposition~\ref{#1}}

\newcommand{\defref}[1]{Definition~\ref{#1}}

\DeclareMathOperator{\Spec}{Spec}
\DeclareMathOperator{\spec}{Spec}

\DeclareMathOperator{\Proj}{Proj}
\DeclareMathOperator{\proj}{Proj}

\DeclareMathOperator{\supp}{Supp}

\DeclareMathOperator{\Img}{Im}

\DeclareMathOperator{\chara}{char}

\DeclareMathOperator{\di}{div}

\DeclareMathOperator{\nr}{nr}
\DeclareMathOperator{\br}{\bar r}
\DeclareMathOperator{\gon}{gon}

\newcommand{\m}{\mathfrak m}

\newcommand{\PP}{\mathbb P}

\newcommand{\Z}{\mathbb Z}

\newcommand{\bbP}{\ensuremath{\mathbb P}}

\newcommand{\bbZ}{\ensuremath{\mathbb Z}}

\newcommand{\cB}{\mathcal B}

\newcommand{\cF}{\mathcal F}

\newcommand{\cO}{\mathcal O}

\renewcommand{\:}{\colon}

\newcommand{\fl}[1]{\left\lfloor #1 \right\rfloor}
\newcommand{\ce}[1]{\left\lceil #1 \right\rceil}

\newcommand{\defset}[2]{{\left\{#1\,\left| \,#2 \right. \right\}}}

\begin{document}

\title{The normal reduction number of two-dimensional cone-like singularities}
\author{Tomohiro Okuma}
\address[Tomohiro Okuma]{Department of Mathematical Sciences, 
Yamagata University,  Yamagata, 990-8560, Japan.}
\email{okuma@sci.kj.yamagata-u.ac.jp}
\author{Kei-ichi Watanabe}
\address[Kei-ichi Watanabe]{Department of Mathematics, College of Humanities and Sciences, 
Nihon University, Setagaya-ku, Tokyo, 156-8550, Japan}
\email{watanabe@math.chs.nihon-u.ac.jp}
\author{Ken-ichi Yoshida}
\address[Ken-ichi Yoshida]{Department of Mathematics, 
College of Humanities and Sciences, 
Nihon University, Setagaya-ku, Tokyo, 156-8550, Japan}
\email{yoshida@math.chs.nihon-u.ac.jp}
\thanks{This work was partially supported by JSPS Grant-in-Aid 
for Scientific Research (C) Grant Numbers, 17K05216, 19K03430}
\subjclass[2010]{Primary 13B22; Secondary 14B05, 14J17}
\keywords{Normal reduction number, two-dimensional singularity, homogeneous hypersurface singularity}

\begin{abstract}
Let $(A, \m)$ be a normal two-dimensional local ring 
and $I$ an $\m$-primary integrally closed ideal with a minimal reduction $Q$. Then we calculate the numbers: 
$\nr(I) = \min\{n \;|\; \overline{I^{n+1}} = Q\overline{I^n}\}, 
\quad \bar{r}(I) = \min\{n \;|\; \overline{I^{N+1}} =
 Q\overline{I^N}, \forall N\ge n\}$, 
$\nr(A)$, and $\bar{r}(A)$, where $\nr(A)$ (resp. $\bar{r}(A)$) is the maximum of  
$\nr(I)$ (resp. $\bar{r}(I)$) for all $\m$-primary integrally closed ideals $I\subset A$. 
Then we have that $\bar{r}(A) \le p_g(A) + 1$, where $p_g(A)$ is the geometric genus of $A$. 
In this paper, we give an upper bound of $\bar{r}(A)$ when $A$ is a cone-like singularity
 (which has a minimal resolution whose 
exceptional set is a single smooth curve) and show, in particular, 
if $A$ is a hypersurface singularity defined by a 
homogeneous polynomial of degree $d$, then $\bar{r}(A)= \nr(\m) = d-1$. 
Also we give an example of $A$ and $I$ so that 
$\nr(I) = 1$ but 
$\bar{r}(I)= \bar{r}(A) = p_g(A) +1=g+1$ for every integer $g \ge 2$.
\end{abstract}

\maketitle
\section{Introduction}

For a Noetherian local ring $(A,\m)$ and an 
$\m$-primary ideal $I$,  
let $\overline{I}$ denote the integral closure, that is, 
$z \in \overline{I}$ if and only if 
$z^n+c_1z^{n-1}+\cdots+c_n=0$ 
for some $n \ge 1$ and $c_i \in I^i$ $(i=1,\ldots,n)$. 
\par 
For a given Noetherian local ring $(A,\m)$ and an 
$\m$-primary integrally closed ideal $I$ (i.e. $\overline{I}=I)$ with minimal 
reduction $Q$, we are interested in the question:

\par \vspace{2mm}
{\bf Question.} What is the minimal number $r$ such that $\overline{I^r} 
\subset Q$ for every $\m$-primary ideal $I$ of $A$ and 
its minimal reduction $Q$? 

\par\vspace{2mm}
One example of this direction is the Brian\c con-Skoda Theorem saying;
\par \noindent
If  $(A,\m)$ is a $d$-dimensional rational singularity (characteristic 
$0$) or an F-rational ring (characteristic $p>0$), then $\overline{I^d} \subset Q$ (cf. \cite{LT}, \cite{HH}).
 
\par \vspace{2mm}
The aim of our paper is to answer this question in the case of normal 
two-dimensional local rings 
using resolution of singularities.
In what follows, we always assume that $(A, \m, k)$ is an excellent normal two-dimensional local ring such that $k$ is algebraically closed and $k\subset A$.

\par
In our previous paper \cite{OWY4}
we defined the notion of two kinds of normal reduction numbers.
For any $\m$-primary integrally closed ideal $I \subset A$ 
(e.g. the maximal ideal $\m$) and its minimal reduction $Q$, 
we define two normal reduction numbers as follows:
\begin{eqnarray*}
\nr(I) &=& \min\{n \in \bbZ_{\ge 0} \,|\, \overline{I^{n+1}}=Q\overline{I^n}\}, \\
\br(I) &=& \min\{n \in \bbZ_{\ge 0} \,|\, \overline{I^{N+1}}=Q\overline{I^N} \; \text{for every $N \ge n$}\}.
\end{eqnarray*}

These are analogues of the reduction number $r_Q(I)$ of an ideal 
$I \subset A$. 
But in general, $r_Q(I)$ is not independent of the choice of a minimal 
reduction $Q$. 
On the other hand, we can show that $\nr(I)$ and $\br(I)$ are 
independent of 
the choice of $Q$ (see e.g. \cite[Theorem 4.5]{Hun}). 
It is obvious by definition that $\nr(I)\le \br(I)$, 
but an example with 
$\nr(I)< \br(I)$ seems to be  \textit{not} known until now. 
We will give a series of examples with $\nr(I)=1$ and 
$\br(I) =p_g(A) +1=g+1$ for all integers $g\ge 2$ in Example \ref{nr<br}.
Also, we define 
\begin{eqnarray*}
\nr(A) &=& \max\{\nr(I) \,|\, 
\text{$I$ is an $\m$-primary integrally closed ideal of $A$}\}, \\
\br(A) &=& \max\{\br(I) \,|\, \text{$I$ is an $\m$-primary integrally closed 
ideal of $A$}\}.
\end{eqnarray*}

\par
We expect that these invariants of $A$ characterize \lq\lq  good" singularities.

\begin{prop}
[{cf. \cite{Li} and \cite[Remark 2.3]{OWY3}}]
The following are equivalent:
\begin{enumerate}
\item $A$ is a rational singularity (i.e., $p_g(A) =0$).
\item $\br(A) =1$.
\item $\nr(A)=1$.
\end{enumerate}
\end{prop}
\begin{proof}
Since there exists an $\m$-primary integrally closed ideal of $A$ which is not a parameter ideal, we have $\nr(A)\ge 1$.  Thus (2) implies (3).
By \cite[Remark 2.3]{OWY3}, $A$ is a rational singularity if and only if every $\m$-primary integrally closed ideal is a $p_g$-ideal.
Furthermore, by \cite[Theorem 4.1]{OWY3}, an $\m$-primary integrally closed ideal $I$ is a $p_g$-ideal if and only if $\br(I)=1$.
Therefore, (1) and (2) are equivalent. 
From \proref{q(nI)} (3) below, we obtain that (3) implies (1), because there exists an $\m$-primary integrally closed ideal $I$ with $q(I)=q(2I)=0$ (cf. \cite[Remark 2.3]{OWY3}).
\end{proof}

\begin{prop}
[{cf. \cite{o.h-ell}}]
If $A$ is an elliptic singularity, then $\nr(A) = \br(A)=2$, 
where we say that $A$ is an elliptic singularity 
if the arithmetic genus of the fundamental cycle
 (see \defref{d:cyc}) on 
a resolution 
of $A$ is $1$. 
\end{prop}

\par
One of the main aims is to compare these invariants 
with geometric invariants (e.g. the geometric genus $p_g(A)$). 
In \cite{OWY2} we have shown that $\br(A)\le p_g(A)+1$. 
But actually, it turns out that we have a much better bound for $\nr(A)$.  

\begin{thm}[\textrm{\cite[Theorem 2.9]{OWY4}}] \label{1.1} 
$p_g(A) \ge \binom{\nr(A)}{2}$.
\end{thm}

\section{The sequence $q(n I)$ and the normal reduction numbers}

Let $(A,\m)$ be an excellent two-dimensional normal local ring 
containing an algebraically closed field $k \cong A/\m$ 
and 
$f\:X \to \spec (A)$ a resolution of singularities with exceptional 
divisor $E:=f^{-1}(\m)$.
We call a divisor supported on $E$ a {\em cycle}.  
A cycle $D$ is said to be positive if $D>0$. 
Let $I = I_Z\subset A$ be an $\m$-primary integrally closed 
ideal {\em represented} by an anti-nef cycle $Z>0$ on $X$, that is, $I\mathcal{O}_X$ is invertible and 
$I\mathcal{O}_X=\mathcal{O}_X(-Z)$.  

For any coherent sheaf $\cF$ on $X$, we write $H^i(\cF)=H^i(X,\cF)$ and $h^i(\cF)=\ell_A(H^i(\cF))$.
\begin{defn} \label{qI}
Put $q(0 I)=h^1(\mathcal{O}_X)$, 
$q(I)= h^1(\mathcal{O}_X(-Z))$ and 
$q(nI)= q(\overline{I^n}) = h^1(\mathcal{O}_X(-nZ))$ for every 
integer $n \ge 1$;
these are independent of the representation of $I$ (\cite[Lemma 3.4]{OWY1}).  
By definition, 
$p_g(A)=q(0I)$.
\end{defn}

\par
We have seen in \S2 of \cite{OWY4} and \S3 of \cite{OWY2} the following results.
\begin{prop}\label{q(nI)}
The following statements hold. 
\begin{enumerate}
\item $0 \le q(I) \le p_g(A);$ and 
\item $q(kI) \ge q((k+1)I)$ for every integer $k \ge 1$ and if $q(nI) = q((n+1)I)$ for some $n\ge 0$, 
then $q(nI) = q(mI)$ for every $m\ge n$. 
Hence $q(nI) = q((n+1)I)$ for every $I$ and $n\ge p_g(A)$.
\item 
For any integer $n \ge 1$, we have 
\[
2 \cdot q(nI) + \ell_A(\overline{I^{n+1}}/Q\overline{I^n})
=q((n+1)I)+q((n-1)I).  
\]
Hence we can describe $\nr(I), \br(I)$ as follows.
\item We have
\begin{eqnarray*}
\nr(I) &=& \min\{n \in \bbZ_{\ge 0} \,|\, 
q((n-1)I)-q(nI)=q(nI)-q((n+1)I) \},\\
\br(I) &=& \min\{n \in \bbZ_{\ge 0}\,|\, q((n-1)I)=q(nI) \}.
\end{eqnarray*}
\end{enumerate}
\end{prop}


\section{The vanishing theorem and the main Theorem}

Our goal is to give an upper bound of $\nr(A)$ and $\br(A)$ 
for cone-like singularities.
For that purpose, we use the vanishing theorem of 
R\"ohr (\cite[Theorem 1.7]{rohr}). 
First we review the fundamental cycle on  a resolution and the computation 
sequence.

  Let $f\: X \to \Spec(A)$ be any resolution of singularity of 
$\Spec(A)$ and 
$E = \bigcup_{i=1}^r E_i$ be the exceptional set of $X$.  

\begin{defn}\label{d:cyc}

\begin{enumerate}
\item A divisor on $X$ is called {\it nef} (resp. {\it anti-nef})  if $D E_i\ge 0$ (resp. $D E_i\le 0$) for every $E_i$.
\item There exists a unique minimal positive anti-nef cycle; 
 we call the cycle {\em the fundamental cycle} of $X$ and write $\bbZ_X$.
\item For a positive cycle $Y$ on $X$, we define an {\em arithmetic genus} $p_a(Y)$ of $Y$ by $p_a(Y)=1-\chi(\cO_Y)=1-h^0(\cO_Y)+h^1(\cO_Y)$. 
By the Riemann-Roch theorem, we have
\[
p_a(Y) = \dfrac{ Y^2 + K_X Y}{2} + 1,
\]
where $K_X$ is the canonical divisor on $X$.
This formula implies
\[
p_a(Y_1+Y_2) = p_a(Y_1)+p_a(Y_2)+Y_1Y_2-1.
\]
It is known that the arithmetic genus of the fundamental cycle is independent of the choice of a resolution. 

\item A sequence of 
cycles $0=Y_0 < Y_1<Y_2 <\ldots < Y_N$ is called a {\em computation sequence} for $\bbZ_X$ if $Y_N$ is anti-nef, $Y_{i+1} = Y_i + E_{j_i}$ for every $i$, $0\le i \le N-1$, $Y_1=E_{j_0}$ is an irreducible component of $E$, $Y_i E_{j_i} >0$ for $i \ge 1$. 
It is easy to show that $Y_N = \bbZ_X$ in this case.
\item We denote by $\cB_X$ the set of positive cycles on $X$ appearing in some computation sequence for $\bbZ_X$.  
If $W$ is a connected subvariety of $E$, then we denote 
by $\bbZ_W$ the fundamental cycle of $W$ (namely, $\bbZ_W$ is the 
minimal cycle supported on $W$ such that $E_i \bbZ_W\le 0$ for every irreducible curve 
$E_i\subset W$).
\item  When $p_g(A)>0$, a positive cycle $C_X$ on $X$ is 
called the {\em cohomological cycle} of $X$ 
if $h^1(\cO_{C_X})= p_g(A)$ and also $C_X$ is the minimal cycle with this property. In \cite{Re}, it is shown that the cohomological cycle exists.
If $W\subset E$ is a reduced connected subvariety with 
$p_a(\bbZ_W)>0$,  we call a positive cycle 
$C_W$ the {\em cohomological cycle} of $W$ if $h^1(\cO_{C_W})$ takes the 
maximal value among the positive cycles supported on $W$ and 
$C_W$ is minimal with this property (cf. \cite[3.4]{o.h-ell}). 
\item For an anti-nef cycle $Z$ on $X$, let $Z^{\bot}=\sum_{ZE_i=0}E_i$.
\end{enumerate}
\end{defn}

Note that if $\{Y_i\}_{i=0}^N$ is a computation sequence for $\bbZ_X$, then $h^0(\cO_{Y_i})=1$ and $h^1(\cO_{Y_i})\le h^1(\cO_{Y_{i+1}})$ for $i\ge 1$ (see \cite[\S 2]{la.me}). In particular, $p_a(Y_i)=h^1(\cO_{Y_i})$ for $i\ge 1$. 

The following theorem holds true in any characteristic (cf. \cite[Ch.~4, Exe.~15]{Re}).

\begin{thm}[R\"ohr's Vanishing Theorem] \label{l:rohr}
Let $D$ be a divisor on $X$.  
Then we have $H^1(X, \cO_X(D)) = 0$ if $YD> 2 p_a(Y) -2$ 
for every $Y \in \cB_X$. 

\end{thm}

\begin{defn}\label{d:cone}
Assume that $(A, \m, k)$ is an excellent normal two-dimensional local ring such that $k$ is algebraically closed and $k\subset A$.  
Let $f_0 \: X_0 \to \Spec(A)$ be the minimal resolution of $\Spec(A)$ and $F$ be the exceptional set of $f_0$.
We call $A$ a {\em cone-like singularity} if $F$ consists of a unique smooth irreducible curve.

The most typical example of cone-like singularities is the localization of a two-dimensional normal graded ring $R = \bigoplus_{n\ge 0} R_n$ generated by $R_1$ over $R_0=k$.  
In fact, the blowing-up of $\Spec(R)$ by the graded maximal ideal is the minimal resolution with exceptional set $\proj (R)$.
\end{defn}

{\bf In the following, we assume that $A$ is a cone-like singularity.}

Let $g$ denote the genus of the curve $F$.

\begin{rem}
We can decompose $f = f_0\circ g$ with 
$g : X \to X_0$ and we denote always $E_0$ the strict transform of $F$ in $X$.  
Note that in this case, the fundamental cycle $\bbZ_X = g^*(\bbZ_{X_0})$ and for every $V \in \cB_X$, $p_a(V)\le p_a(\Z_X)=g=p_a(E_0)$; in fact, 
we  have either 
\begin{enumerate}
\item [(i)] $E_0\le V$ and $p_a(V) =g$, or 
\item [(ii)]  $\supp(V)$ is a tree of $\bbP^1$ and $p_a(V) =0$.  
\end{enumerate}
\end{rem}

Under our assumption, we have the following  vanishing theorem.
\begin{lem}\label{rohr} 
Let $D$ be a nef divisor on $X$.  If $D E_0>2g-2$, 
then $H^1(\cO_X(D))=0$.
\end{lem}

\begin{proof}
By Theorem \ref{l:rohr}, 
$H^1(\cO_X(D))=0$ if $D V> 2p_a(V)-2$ for 
any positive cycle $V\in \cB_X$. 
If $E_0\le V$, then $p_a(V) = g$ and $DV \ge DE_0 > 2g-2$ and 
otherwise $DV \ge 0 > 2 p_a(V) -2 = -2$ and we have our conclusion.
\end{proof}

The following proposition plays an important role for our main theorem.

\begin{prop}\label{ZE=0}
Let $I=I_Z$ be an $\m$-primary integrally closed ideal of $A$ 
represented by a cycle $Z$ on a resolution $X$ 
of $\Spec(A)$ and assume that $ZE_0=0$. 
Let $B$ be the maximal reduced connected cycle containing $E_0$ such that 
$B\le Z^{\bot}$. 
Let $Z_B=\bbZ_B$ and $C_B$ be the cohomological cycle on $B$.
Then $-Z_BE_0\ge -\bbZ_X^2=-F^2$, and for every integer $s> (2g-2)/(-Z_BE_0)$ we have 
\[
H^1(\cO_X(-s(Z+Z_B)))=0, \quad 
H^1(\cO_X(-sZ))\cong H^1(\cO_{sZ_B}(-sZ))\cong H^1(\cO_{C_B}).
\]
\end{prop}

\begin{proof}
We note that $Z_B\in \cB_X$ 
and $\bbZ_X - Z_B$ does not contain $E_0$. 
It is known that $\bbZ_X=g^*\bbZ_{X_0}$. 
Then the first assertion follows from that $Z_BE_0\le \bbZ_X E_0 =(g^{*}(F))^2=F^2$.
Next we show that $-(Z+Z_B)$ is nef.
If $E_i\le B$, then $(Z+Z_B)E_i=Z_BE_i\le 0$.
If $E_i \cap B=\emptyset$, then $(Z+Z_B)E_i=ZE_i\le 0$.
Assume that $E_i\le E-B$ and $BE_i>0$. Then $Z_B+E_i$ appears in a computation sequence since $Z_BE_i>0$, and 
\[
g=p_a(Z_B+E_i)=p_a(Z_B)+p_a(E_i)+Z_BE_i-1=g+Z_BE_i-1.
\]
Thus $Z_BE_i =1$. Since $ZE_i<0$ by the definition of $B$, we have that $(Z+Z_B)E_i\le 0$. 
Hence we obtain that $Z+Z_B$ is anti-nef.
Since $-s(Z+Z_B)E_0=-sZ_BE_0>2g-2$, the vanishing follows from Lemma \ref{rohr}.
From the exact sequence
\[
0 \to \cO_X(-s(Z+Z_B)) \to \cO_X(-sZ) \to \cO_{sZ_B}(-sZ) \to 0,
\]
we have $H^1(\cO_X(-sZ))\cong H^1(\cO_{sZ_B}(-sZ))$.
Since there exists a function $h\in I_Z$ such that $\di_X(h)=Z+H$, where $H$ is the proper transform of $\di_{\spec(A)}(h)$, 
we have  $\cO_{sZ_B}(-sZ)\cong \cO_{sZ_B}(-s\di_{X}(h))\cong \cO_{sZ_B}$ since $HB=0$.
If $s'>s$, then $h^1(\cO_{sZ_B})\le h^1(\cO_{s'Z_B})$ and $h^1(\cO_X(-sZ))\ge h^1(\cO_X(-s'Z))$ by  \proref{q(nI)} (2).
Therefore, $h^1(\cO_X( -sZ))$ is stable for $s > (2g-2)/(-Z_BE_0)$.
Thus we have  $h^1(\cO_X(-sZ)) = h^1(\cO_{sZ_B}) =  h^1(\cO_{C_B})$ by \cite[3.4]{OWY2}. 
\end{proof}

Before stating our main theorem, we prepare some notations and terminologies.

\begin{defn} \label{gon}
Let $C$ be a smooth curve. The {\em gonality} of the curve 
$C$ is the minimum of the degree of surjective morphisms 
from $C$ to $\PP^1$, and denoted by $\gon(C)$.
\end{defn}


Let $\fl{x}$ denote the floor (or, integer part) of a real number $x$.

\begin{thm}\label{main}  
Let $A$ be a cone-like singularity 
and let $I = I_Z$ be an $\m$-primary integrally closed ideal of $A$ represented by a cycle $Z$ on the resolution $X$.  
Let $E_0$ be the unique curve on $X$ with genus $g >0$ and 
let $d=-\bbZ_X^2=-F^2$.
Then we have the following.
\begin{enumerate}
\item If  $ZE_0=0$, then   
$\br(I)\le \fl{(2g-2)/d}+2$.
\item If $ZE_0<0$, then 
$\br(I)\le \fl{(2g-2)/\gon(E_0)}+2$.
\end{enumerate}
\end{thm}
\begin{proof}
(1) By Proposition \ref{ZE=0}, 
$q( sI)= q((s+1)I )$ for $s > (2g-2)/d$. 
Hence we have 
$\br(I)\le \fl{(2g-2)/d}+2$.\par

(2) Let $s=-ZE_0$. Since $\cO_X(-Z)$ is generated, 
there exist sections $\sigma_1, \sigma_2\in H^0(\cO_{E_0}(-Z))$ which determine a surjective morphism $\phi\: E_0\to \PP^1$ of degree $s$.
Hence $s\ge \gon(E_0)$, and 
$d':=\fl{(2g-2)/\gon(E_0)}+1>(2g-2)/s$. 
By \lemref{l:rohr}, we have $q(d'I)=0$. 
\end{proof}

Now we will give an example of $A$ and $I$ with
\[ \nr(I) = 1\quad \text{ and} \quad \br(I) = \br(A) = p_g(A) +1.\]  

\begin{ex}\label{nr<br}
Let $C$ be a hyperelliptic curve with $g(C)=g\ge 2$
 and $D_0$ a divisor on $C$ which is the pull-back of a point via the double cover $C\to \PP^1$. 
Let $b\in \Z_{>0}$, $D=bD_0$, and $R=R(C,D)= 
\bigoplus_{n \ge 0} H^0(X, \cO_C(nD))$. 
We write as $h^i(D)=h^i(\cO_C(D))$. 
Recall that
\begin{itemize}
\item $h^0(nD_0) = n+1$ and $h^1(nD_0) = g-n$ if $n\le g-1$, and 
\item $h^0(nD_0)=2n+1-g$ and $h^1(nD_0) =0$ if $n\ge g$.
\end{itemize} 
Hence    
$p_g(R) = \sum_{0\le bn \le g-1} (g-bn)$   by \cite[Theorem 5.7]{Pi}.  

Let $Y\to \spec R$ denote the minimal resolution with exceptional set $F\cong C$; we may regard $F=C$.
Then $\cO_F(-F)\cong \cO_C(D)$ and $-F^2=2b$. If we take  a general  element $h\in I_F$, then $\di_Y(h)=F+H$, 
where $H$ is the non-exceptional part and $F\cap H$ consists of distinct $2b$ points $P_1, \dots, P_{2b}$.  
Assume that $P_1+P_2\sim D_0$, and let $X \to Y$ be the blowing-up with center $\{P_3, \dots, P_{2b}\}$ ($X=Y$ if $b=1$) and $Z$ the exceptional part of $\di_X(h)$.
Then $\cO_X(-Z)$ is generated since a general element of $R_2$ has no zero on $H$. 
We have $-ZE_0=2$ and $-Z^2 = 4b -2$.

Since  $\cO_X(-(g-1)Z)\otimes_{\cO_X} \cO_{E_0} \cong \cO_{E_0}(K_{E_0})$, we have $h^1(\cO_X(-(g-1)Z))\ge h^1(K_{E_0})=1$ 
and $H^1(\cO_X(-gZ))=0$ by \lemref{l:rohr}.
Hence $\br(I_Z)=g+1=\fl{(2g-2)/\gon(E_0)}+2$.

Now, let us assume $b\ge g$.  Then we have $p_g(A) = g$.  
Since $\br(I_Z)>1$, $I_Z$ is not a $p_g$-ideal, and thus 
 $q(I_Z) \le p_g(A) - 1 = g-1$. 
From $q((g-1)I_Z) = 1$ and $q(gI_Z) =0$, we must have $q(nI_Z) = g -n$ for $n\le g$ by 
Proposition \ref{q(nI)} (2).  Hence we have $\nr(I_Z) =1$ by Proposition \ref{q(nI)} (4). 
\end{ex}


The following is a ring-theoretic expression of an example similar to Example \ref{nr<br} with $b=g$, which was found in our attempt to translate Example \ref{nr<br} into ring-theoretic language.

\begin{ex}\label{Vero}  Let $g$ be a positive integer $\ge 2$ and put
\[
R =  k[X,Y,Z]/ ( X^2 + Y^{2g+2} + Z^{2g+2}).
\]
Assume that $\chara k$ does not divide $2g+2$.
Then $R$ is a normal graded ring with 
$(\deg X, \deg Y, \deg Z) = (g+1, 1, 1)$.
Let $A$ be the $g$-{th} Veronese subring of $R$:
\[
A = R^{(g)} = k[y^g, y^{g-1}z, \ldots, z^g, xy^{g-1}, xy^{g-2}z, \ldots xz^{g-1}],
\]
where $x,y,z$ denotes, respectively, the image of $X,Y,Z$ in $R$.
Note that $C:=\proj R$ is a hyperelliptic curve with $g(C)=g$ and $R=R(C,D_0)$ with $D_0$ as in Example \ref{nr<br},  and we have $A=R(C,gD_0)$. 
Since $g > a(R) = g-1$, we have $a(A) =0$ and $p_g(A) = g$.
We put 
\[
I = ( y^g, y^{g-1}z, A_{\ge 2})\subset A,
\]
 the ideal generated by 
$y^g, y^{g-1}z\in A_1$ and all the elements of $A_i$ with $i\ge 2$,
 and 
\[Q = (y^g- z^{2g}, y^{g-1}z)\subset A.\] 
Then we can show the following.
\begin{enumerate}
\item $I$ is integrally closed and $Q$ is a minimal reduction of $I$; in fact, $I^2=QI$.
\item $\ell_A(A/I) = g$ and $e(I) = 4g -2$.
\item $\overline{I^{n+1}} = Q\overline{I^{n}}$ for $n\ge 1$ and $n \ne g$.
It follows that $\overline{I^{n}}=I^n$ for $n \le g$ by (1).
\item $xy^{g^2-1} \in  \overline{I^{g+1}}$ and $\not\in   Q\overline{I^{g}}$.
\item $q(nI) = g-n$ for $n\le g$ and $q(nI) = 0$ for $n\ge g$. 
In particular, $\nr(I) =1$ and $\br(I) = g+1$.
Since we know $\br(A) \le p_g(A) +1$, this shows that $\br(A) = g+1$.
Since $\ell_A(\overline{I^{g+1}}/Q \overline{I^{g}})=1$ by \proref{q(nI)} (3), it also follows that 
$\overline{I^{g+1}}=I^{g+1}+(xy^{g^2-1})$.
\end{enumerate}

It is easy to see that $I$ is integrally closed and that $\dim_k A/I= g$.  

If we put $Q_0 =  (Y^g- Z^{2g}, Y^{g-1}Z)\subset A_0:= k[Y,Z]^{(g)}$, then we see 
that $(Y,Z)^{3g} \cap A_0 \subset Q_0$ 
and we can take 
\[\{1, y^g=z^{2g}, y^{g-2}z^{2}, \ldots, z^g, z^{g+2}y^{g-2}, \ldots , z^{2g-1}y\}\]
as a basis of  $A_0/Q_0$ and hence $\dim_k  A_0/ Q_0= 2g-1$ , which implies
\[\ell_A( A /Q) = 4g-2.\]

Then we will show that $I^2 = QI$.  
Note that $I^2$ is generated by 
\[
\{y^{2g}, 
y^{2g-1}z, y^{2g-2}z^2\}, \ \ (y^g, y^{g-1}z)A_2, \ \ \text{and } A_4.
\]
   We have seen 
$A_4\subset Q A_{\ge 2}\subset QI$ and if $h \in A_2$, then 
$y^g h = (y^g -z^{2g}) h + z^{2g}h \in QI$ since $ z^{2g}h \in A_4 \subset QI$.
Hence 
$I^2=QI$.

Since 
\[
( xy^{g^2-1})^2= (y^{2g+2}+ z^{2g+2})(y^{2g^2-2}) = (y^g)^{2g+2}+(y^{g-1}z)^{2g+2}  \in I^{2g+2},
\]
we see that $xy^{g^2-1}\in \overline{I^{g+1}}$.  
But we can also see that $xy^{g^2-1}\not\in Q \overline{I^g}$ as follows.

First, we prove

\begin{clm}\label{barI^n}
  For every $n\ge 1$, we have the following:
 \begin{enumerate}
 \item[(a)] If $f_0=f_0(y,z)$ is a homogeneous polynomial of degree $ng$ of $y,z$, 
 then $f_0 \in \overline{I^n}$ if and only if $f_0\in I^n$.
 \item[(b)] Let $f_1=f_1(y,z)$ be a  homogeneous polynomial of degree $(n-1)g -1$ of $y,z$.
  If $xf_1\in \overline{ I^n}$, then  $n\ge g+1$  and the highest power of $z$ appearing in $f_1$ is at most $n - (g+1)$.
Therefore, $xf_1\not\in \overline{ I^n}$ if $n\le g$.
\item[(c)] If $n\le g$, then $\overline{I^n}\cap A_n= I^n\cap A_n$.
\end{enumerate}
\end{clm}
\begin{proof}
(a) 
If $f_0$ is integral over $I^n$, then there is an integral equation
\[f_0^s + c_1 f_0^{s-1} + \cdots + c_j f_0^{s-j} + \cdots + c_s = 0\]
with  $c_j \in I^{nj}\cap A_{nj}$.  Now, we can write 
\[ c_j = c_{j,0} + x c_{j,1}, \] 
where $c_{j,0}$ (resp.  $c_{j,1}$) is a homogeneous polynomial of degree $njg$ (resp. $(nj-1)g -1$)
of $y,z$.  Since $A_n = k[y,z]_{ng} \oplus x k[y,z]_{(n-1)g -1}$ as $k[y,z]^{(g)}$-module,
we have 
\[f_0^s + c_{1,0} f_0^{s-1} + \cdots + c_{j,0} f_0^{s-j} + \cdots + c_{s,0} = 0\]
and we have our result since for the ideal $I_0 :=(y^g, y^{g-1}z)+(k[y,z]^{(g)})_{\ge 2}$ in $k[y,z]^{(g)}$,  
$I_0^n$ is integrally closed in $k[y,z]^{(g)}$.
\par
(b)
Suppose that $xf_1$ is integral over $I^n$. Then 
\[(xf_1)^2 = f_1^2(y^{2g+2}+ z^{2g+2})\]
should be integral over $I^{2n}$ and included in $I^{2n}$ by (a).
Hence the highest power of $z$ appearing in 
$f_1^2(y^{2g+2}+ z^{2g+2})$ is at most $2n$,
 and then $2g+2\le 2n$ and the highest power of $z$ appearing in $f_1^2$ is at most $2(n-g-1)$.

(c) 
 Let $f_0$ (resp. $f_1$) be a homogeneous polynomial of degree $ng$ (resp. $(n-1)g-1$) in 
$y,z$.  We assume that $f_0 + xf_1 \in \overline{I^n}\cap A_n$ with $n\le g$.
By (a), it suffices to show if $f_0 + xf_1\in \overline{I^n}\cap A_n$, then $f_1=0$.
Since $y^{ng},\ldots , y^{ng-n}z^n \in I^n$, we may assume that $f_0 = z^{n+1}\phi$ for some $\phi \in k[y,z]$.
\par 
Let $\sigma : A \to A$ be the automorphism with $\sigma(x)=-x$ and fix $y$ and $z$. 
Then since $I$ is stable 
under $\sigma$, if $f_0 + xf_1 \in \overline{I^n}$, then $f_0- xf_1 \in \overline{I^n}$.  
Thus we should have 
\[f_0^2 - f_1^2 (y^{2g+2} + z^{2g+2}) \in \overline{I^{2n}}. \] 

Now write   
\[f_0 = \sum_{i= n+1}^{ng} a_{i}y^{ng -i}z^i \quad {\mbox and} \quad f_1= \sum_{j=0}^{(n-1)g -1} b_jy^{(n-1)g -1-j}z^j.\]  
Since the highest power of $z$  appearing in $f_0^2 - f_1^2 (y^{2g+2} + z^{2g+2})$ should be at most $2n$ by (a), if $u \ge 0$ is the biggest such that $b_u\ne 0$, then $a_{u+g+1}$ should be the biggest such that $a_{u+g+1}\ne 0$ and 
$a_{u+g+1}^2 = b_u^2$.\par
 
Let $f = f_0+xf_1$ and let
\[ f^s + f^{s-1} h_1+ \cdots + h_s = 0 \]
be an integral equation of $f$ over $I^n$ so that $h_j \in I^{nj}$ for every $j$. 
Now we will deduce a contradiction. 

The highest power of $z$ in $f^s$ is $s(u + g +1)$. But for $j>0$, since  
the highest power of $z$ in $h_j$ is at most $nj$, we conclude every power of $f^{s-j}h_j$ is 
strictly less than   $s(u + g +1)$.  Thus we have a contradiction. 
\end{proof}
  
Now, let us return to the proof of $xy^{g^2-1}\not\in Q \overline{I^g}$.
We can write 
\[xy^{g^2-1}= xy^{g^2-g-1}(y^g-z^{2g}) + (xy^{g^2-2g}z^{2g-1})(y^{g-1}z).\]
Since $Q$ is generated by regular sequence, any expression of  $xy^{g^2-1}$ as an element of $Q$
should be of the form 
\[xy^{g^2-1}= (xy^{g^2-g-1} + h y^{g-1}z )(y^g-z^{2g}) + (xy^{g^2-2g}z^{2g-1} - (y^g-z^{2g})h)(y^{g-1}z)\]
for some $h\in A$.

Assume $xy^{g^2-g-1} + h y^{g-1}z$ and $xy^{g^2-2g}z^{2g-1} - (y^g-z^{2g})h
\in \overline{I^g}$.  Since $\overline{I^g}$ is a homogeneous ideal, we may assume that $h \in A_{g-1}$.  
By Claim \ref{barI^n} (c), we have $xy^{g^2-g-1} + h y^{g-1}z$ and $y^gh$ are elements of $I^g$.
Therefore, we have $xy^{g^2-g-1}\in I^g$; however, it follows that $xy^{g^2-g-1}\not\in \overline{I^g}$ by Claim \ref{barI^n} (b).
Thus our proof that $xy^{g^2-1}\not\in Q \overline{I^g}$ is completed. 

\par 
In order to see (3) and (5), we suppose $q(kI)=q((k+1)I)$ for some $k \le g-1$. 
Then $q((g-1)I)=q(gI)=q((g+1)I)$ by \cite[Lemma 3.1]{OWY2}. 
This implies $\overline{I^{g+1}}=Q\overline{I^g}$, which contradicts that 
$xy^{g^2-1} \in \overline{I^{g+1}} \setminus Q\overline{I^{g}}$. 
Hence we have 
\[
g=p_g(A)=q(0\cdot I)> q(1 \cdot I)> \cdots > q((g-1)I)> q(gI) \ge 0.
\]
It follows that $q(nI)=g-n$ for all $n \le g$ and $q(nI)=0$ for all $n \ge g$. 
Furthermore, the other assertions follows from \cite[Lemma 2.7]{OWY4}. 
\end{ex}

\section{The homogeneous case}\label{s:hmg}

In this section we treat a two-dimensional standard normal graded ring $R = \bigoplus_{n\ge 0} R_n$ over an algebraically closed field $R_0=k$.
We can also express $R$ as 
\[ R = R(C,D) = \bigoplus _{n\ge 0} H^0( C, \cO_C( nD)), \]
where $C=\Proj(R)$ and $D$ is  
a very ample divisor on $C$ such that $\cO_C(D) \cong \cO_C(1)$.  We write $\m = R_+ = \bigoplus_{n>0} R_n$.
Assume that $f\: X\to \spec R$ is the minimal resolution.
Then $f$ is the blowing-up by the maximal ideal $\m$, $E\cong C$, and $\m$ is represented by $E$. 
We know in this case $\m^n$ is integrally closed for every $n>0$; therefore $\m^n=I_{nE}$.  

The invariant $a(R)$ of $R$ is given by (see \cite[(3.1.4)]{G-W}, \cite[\S 2]{KeiWat-D})
\[a(R) = \max\{n \;|\; H_{\m}^2(R)_n \ne 0\} =  \max\{n \;|\; h^1(\cO_C(nD) \ne 0\}.\]
Since we are interested in non-rational singularities, we always assume $a(R) \ge 0$.

By \cite{Pi} and  \cite[\S 6]{tki-w}, we have
\[p_g(R) = \sum_{n=0}^{a(R)} h^1(\cO_C(nD))\quad \text{ and } \; q(k\m) = \sum_{n\ge k} h^1(\cO_C(nD)).\]

If $Q$ is a minimal reduction of $\m$ generated 
by elements of $R_1$, since $a(R/Q) = a(R) +2$ (cf. \cite[(3.1.6)]{G-W}), we have 
\begin{equation}\label{r(m)-homog}
\m^{a(R)+2} \ne Q\m^{a(R)+1} \ \ \text{ and} \ \  \nr(\m) = a(R) + 2 = \br(\m), 
\end{equation}
the latter equality holds from $q(a(R)\m) >0$ and $q((a(R)+1)\m) =0$.

We shall show that if $R$ is a hypersurface or complete intersection satisfying certain conditions, 
then we have $\br(R) =\br(\m) = \nr(\m)=\nr(R)$.

\subsection{Hypersurfaces}
Assume that $R = k[X,Y,Z]/ (f)$, where $f$ is a homogeneous polynomial of degree $d\ge 3$ with an isolated singularity.
Let $Y\to \spec R$ denote the minimal resolution with exceptional set $F$.
Then $F=\{f=0\}\subset \PP^2$, $g:=g(F)=(d-1)(d-2)/2$, and $a(R)=d-3$.
Let $D=-F|_F$. Then $\deg D=d$ and $R=R(F, D)$. 
Let $P^i(R,t)=\sum_{n\ge 0}h^i(\cO_F(nD))t^n$ and $p^{\chi}(R,t)=P^0(R,t)-P^1(R,t)$. 
Then we have
\begin{align*}
P^0(R,t) &=\frac{1-t^d}{(1-t)^3}, \\
P^{\chi}(R,t)&=\sum_{n\ge 0}\chi(\cO_F(nD))t^n
=\sum_{n\ge 0}(1-g+md)t^m \\
&=\frac{1-g}{1-t}+\frac{dt}{(1-t)^2}
=\frac{1-g+(g+d-1)t}{(1-t)^2}, \\
P^1(R,t)&=P^0(R,t)-P^{\chi}(R,t)=\sum_{m=0}^{a(R)}\binom{d-1-m}{2}t^m.
\end{align*}

The following theorem is one of the main goals of this paper. 

\begin{thm}\label{t:brHmg}
$\nr(\m)=\br(\m)=\nr(R)=\br(R)=d-1=a(R)+2$.
\end{thm}  
\begin{proof}  Since $a(R) = d-3$, we have seen that $\nr(\m) = \br(\m) = a(R) + 2 = d-1$ by \eqref{r(m)-homog}.  
Hence it is sufficient to show that $\br(R)\le d-1$.

By Namba's theorem\footnote{This is also called Max Noether's theorem} \cite[Theorem 2.3.1]{Namba767} (see \cite[Appendix]{Hom-fun} for any characteristic), we have $\gon(F) =d-1$.
By \thmref{main}, we have 
\[
\br(I_Z)\le \fl{(2g-2)/(d-1)}+2=\fl{d-2-2/(d-1)}+2=d-1.   \qedhere
\]
\end{proof}

\begin{rem}\label{q(nm)}
In this case we have 
\[
q(n\m)=\sum_{n\le i \le a(R)} \binom{d-1-i}{2}
=\binom{d-n}{3}.
\]
\end{rem}

\begin{ex}\label{L+mr}  We will give a series of examples with various values of $q(I_Z)$. 
\par
Let $L\in R$ be a general linear form and  let 
$\{P_1, \ldots P_d\}=\{L=f=0\}\subset \PP^2$.
We write $\di_Y(L)=F+\sum_{i=1}^dH_i$, where $F\cap H_i=\{P_i\}$.
Let $X_0=Y$ and $\phi_i\:X_i\to X_{i-1}$ be the blowing up with center the intersection of the exceptional set and the proper transform of $\sum_{i=1}^dH_i$.  Let $E^{(r)}$ denote the exceptional set  of $X_r\to \spec R$.
For every $i$, by abuse of notation, we denote by $H_i$ (resp. $E_0$) the proper transform of $H_i$ (resp. $F$), and $E_{i,j}$ the proper transform of the exceptional curve of the $j$-th blowing up at $E^{(j-1)}\cap H_i$. 
Then $E^{(r)}$ ($r\ge 1$) is star-shaped and expressed as follows:
\[ 
  E^{(r)}= E_0+\sum_{i=1}^d \sum _{j=1}^r E_{i,j}, \quad 
E_{i,1}^2=\cdots = E_{i, r-1}^2 = -2, \quad  E_{i,r}^2 = -1.
\]
We denote by $Z_r$ the exceptional part of $\di_{X_r}(L)$, namely,
\[
Z_r = E_0 + \sum_{i=1}^d ( 2E_{i,1} + 3 E_{i,2} +\dots + r E_{i, r-1}  +(r+1)  E_{i,r} ).
\]
\end{ex}

\begin{defn}
For a graded ring $S=\bigoplus_{i\ge 0}S_i$, let $S_{\ge m}=\bigoplus_{i\ge m}S_i$.
\end{defn}

\begin{prop}
Let $L_{r+1}\in R$ be a general $(r+1)$-form and $Q = (L, L_{r+1})$. 
We have the following.
\begin{enumerate}
\item We have $I_{Z_r} = (L) + \m^{r+1}$ and that $Q$ is a minimal reduction of $I_{Z_r}$. Furthermore, $I_{Z_r}^s$ is integrally closed for every $s\ge 1$.
\item If $r \ge d - 2 = a( R) + 1$, then $I_{Z_r}$ is a $p_g$-ideal, namely, $q(I_{Z_r})=p_g(R)$.
\item We have $q( I_{Z_r}) = q( I_{Z_{r-1}}) + d-r-1$ for $1\le r \le d-1$, where $I_{Z_0}=\m$.
Thus $q( I_{Z_r}) =\binom{d-1}{3}+r(2d-r-3)/2$.
\item We have
\[
\ell_R( I_{Z_r}^s / Q   I_{Z_r}^{s-1} ) =\ell_R\left((R/Q)_{\ge s(r+1)}\right).
\]
Therefore, we have that $I_{Z_r}^s  =  Q   I_{Z_r}^{s-1}$ 
if and only if $s(r+1) \ge d + r$, and that 
$\nr(I_{Z_r}) = \br(I_{Z_r}) = \ce{\frac{d+r}{r+1}}-1=\ce{\frac{d-1}{r+1}}$.
\item If $s\ge \br(I_{Z_r})-1$, we have $q ( s   I_{Z_r} ) = p_g ( R_{d,r})$, where $R_{d,r}$ is the singularity obtained by blowing down $Z_r^{\bot}$. Then $R_{d,r}$ is obtained from $R(E_0,D)$, where $D = ( 1 + 1/{ r} ) ( P_1+\ldots + P_d)$.
\end{enumerate}
\end{prop}
\begin{proof}
Let $S=k[X,Y,Z]/(f,L,L_{r+1})$. Then  the Hilbert series of the Artinian ring $S$ is $(1-t^d)(1-t^{r+1})/(1-t)^2$, and $\ell_R(R/Q)=\dim_k(S)=d(r+1)$. 

(1)
Let $I=(L, \m^{r+1})=(L) + \m^{r+1}$.
Then $I^s$ is integrally closed for every $s\ge 1$ by \lemref{l:app}. 
We have $I \subset I_{Z_r}$, and $\cO_{X_r}(-Z_r)$ is generated by the elements $L$ and $L_{r+1}$. Therefore, $Q$ is a minimal reduction of $I$ and $I=I_{Z_r}$.

(2)
On $X_0$, we have $K_{X_0}=-(d-2)F$. Hence the cohomological cycle $C_r$ on $X_r$ is given as follows (cf. \cite[Proposition 2.6]{OWY3}):
\[
C_r=(d-2)E_0+\sum_{i=1}^d \sum_{j=1}^r a_jE_{i,j}, \quad 
a_j=\max\{d-2-j,0\}.
\]
Therefore, for $r\ge d-2$, we have $Z_rC_r=0$ and 
$\cO_{C_r}(-Z_r)\cong \cO_{C_r}$.
Hence $Z_r$ is a $p_g$-cycle by 
\cite[3.10]{OWY1}.

(3)
Let us consider the exact sequence
\[
0 \to \cO_{X_r}(-Z_r) \to \cO_{X_r}(-\phi_r^*Z_{r-1}) \to \cO_{D_r} \to 0,
\]
where $D_r=\sum_{i=1}^d E_{i,r}$.
Note that $H^0(\cO_{X_r}(-\phi_r^*Z_{r-1}))=H^0(\cO_{X_{r-1}}(-Z_{r-1}))=I_{Z_{r-1}}$.
Then 
\[
\Img\left(H^0(\cO_{X_r}(-\phi_r^*Z_{r-1})) \to H^0(\cO_{D_r}) \right)
\cong (L,\m^r)/(L,\m^{r+1}) \cong k[X,Y]_r\cong k^{r+1}.
\]
Since $H^1(\cO_{D_r})=0$, from the long exact sequence, 
 we obtain $q(I_{Z_r})=q(I_{Z_{r-1}})+d-(r+1)$.

(4)
We have 
\[
I_{Z_r}^s=L(L, \m^{r+1})^{s-1} + \m^{s(r+1)}, \quad
QI_{Z_r}^{s-1}=L(L, \m^{r+1})^{s-1} + L_{r+1}\m^{(s-1)(r+1)}.
\]
Therefore, 
\[
\ell_R( I_{Z_r}^s / Q   I_{Z_r}^{s-1} ) 
=\dim_k(S_{\ge s(r+1)})
=\dim_k\left((R/Q)_{\ge s(r+1)}\right).
\]
Clearly, $I_{Z_r}^s  =  Q   I_{Z_r}^{s-1}$  if and only if 
$s(r+1) > d + r -1$.
The last assertion follows from the definition of $\nr$ and $\br$.

(5)
The first assertion follows from (4) and \cite[Proposition 3.4]{OWY2}.
Let $H=P_1+\cdots + P_d$. Then we have
$Z_r^{\bot}=E_0+\sum_{i=1}^d \sum _{j=1}^{r-1} E_{i,j}$ and $\cO_{E_0}(-E_0)\cong \cO_{E_0}(\di_{X_r}(L)-E_0)\cong \cO_{E_0}(2H)$.
Hence $D=(2H-\frac{r-1}{r}H)|_{E_0}=(\frac{r+1}{r}H)|_{E_0}$ 
(cf. \cite{Pi}, \cite[\S 6]{tki-w}).
\end{proof}

\begin{rem}
By \proref{q(nI)} (3), we have
\[
\ell_R( I_{Z_r}^s / Q   I_{Z_r}^{s-1} ) = \left( q( s I_{Z_r}) + q( (s-2) I_{Z_r})\right) - 2 q( (s-1) I_{Z_r} ).
\]
\end{rem}

\subsection{Complete intersections}
Assume that $R = k[X_0, \dots, X_n]/ (f_1, \dots, f_{n-1})$, where each $f_i$ is a homogeneous polynomial of degree $d_i>1$, and $\spec R$ has an isolated singularity at $\m$. 
Then $a(R)=\sum_{i=1}^{n-1}d_i-n-1$. Let $d=\prod_{i=1}^{n-1}d_i$.

Let $Y\to \spec R$ denote the minimal resolution with exceptional set $F$.
Then $F\cong \{f_1=\cdots =f_{n-1}=0\}\subset \PP^n$. 
By adjunction, $g(F)=d\cdot a(R)/2+1$.
By \eqref{r(m)-homog}, we have $\bar r (\m)=a(R)+2$.

\begin{thm}
Assume that $d_1\le \dots \le d_{n-1}$. Then 
\[
\br(R)\le a(R)+\fl{a(R)/(d_1-1)}+2.
\]
\end{thm}
\begin{proof}
By \cite[4.12]{Laz-LecLS}, $\gon(F) \ge (d_1-1)d_2\cdots d_{n-1}=d(d_1-1)/d_1$.
By Theorem \ref{main}, we have 
\[
\br(I_Z)\le \fl{d a(R) d_1/d(d_1-1)}+2=\fl{a(R) +a(R) /(d_1-1)}+2.  \qedhere
\]
\end{proof}

\begin{ex}
If $n=3$, $d_1=d_2=2$, then $\nr(\m)=\br(R)= a(R)+2$.
\end{ex}

\section{Appendix}\label{s:app}

We give a lemma showing that certain ideals of special type which appeared in \S \ref{s:hmg} are integrally closed.

Let $R = \bigoplus_{n\ge 0} R_n$ be a normal graded ring of dimension $d\ge 2$ which is finitely generated over a field $k=R_0$ and $\m= \bigoplus_{n\ge 1} R_n$.
We fix positive integers $N>m$ and a homogeneous element $f\in R_m$ such that $(f)=\sqrt{(f)}$.   
Then we prove the following.

\begin{lem} \label{l:app}
The ideal $I:=(f)+R_{\ge N}$ is integrally closed. 
Suppose that there exist valuations $v_1, \dots, v_p$ of the quotient field of $R$ such that 
\[
I=\defset{x\in R}{v_i(x)\ge v_i(f), \; 1\le i \le p}.
\]
Then $I^s$ is integrally closed for every $s\ge 1$.
\end{lem}
\begin{proof}
Let $g\in R$ be integral over $I$. We show that $g\in I$.
Since the integral closure of a homogeneous ideal is homogeneous, 
we may assume that $g$ is a homogeneous element of degree $t$ with $m\le t<N$.
There exists a positive integer $u$ and $a_j\in I^j\cap R_{t\cdot j}$ such that $g^u+a_1g^{u-1}+\cdots +a_u=0$.
Since $t<N$, we see that $a_j\in (f)$ for every $j$. 
Therefore, $g^u\in (f)$. 
Hence $g\in \sqrt{(f)}=(f)\subset I$.

Let $s\ge 2$. 
Suppose that $g\in R_t$ is integral over $I^s$ and $ms\le t<Ns$.
There exists a positive integer $u$ and $a_j\in I^{sj}\cap R_{t\cdot j}$ such that $g^u+a_1g^{u-1}+\cdots +a_u=0$.
By the argument above, we have $g\in (f)$ again.
We write $g=fg'$, $g'\in R$.
Then we have  $v_i(g')=v_i(g)-v_i(f) \ge (s-1)v_i(f)$. 
Therefore, $g'$ is integral over $I^{s-1}$. Hence we obtain the claim by induction on $s$.
\end{proof}

\begin{acknowledgement} 
 The authors thank the referee for careful reading of the paper and 
helpful comments.
\end{acknowledgement}

\end{document}